\documentclass[letterpaper, 10 pt, conference]{ieeeconf}  
\usepackage{graphicx}   
\usepackage{mathrsfs}
\usepackage{amsmath}
\usepackage{xspace}
\usepackage{algorithm}
\usepackage{algorithmic}
\usepackage{xcolor}
\usepackage{amssymb}
\usepackage{amssymb}
\usepackage{hyperref}
\newtheorem{remark}{Remark}

\newtheorem{assumption}{Assumption}
\newtheorem{definition}{Definition}
\newtheorem{theorem}{Theorem}

\newcommand{\GR}[1]{\textcolor{black}{#1}}

\IEEEoverridecommandlockouts                              
\title{\LARGE \bf
Constrained Optimization on Matrix Lie Groups via
Interior-Point Method
}

\author{Aclécio J. Santos$^{1}$, Jean C. Pereira$^{2}$, Guilherme V. Raffo$^{1,3}$
\thanks{\copyright 2026. This manuscript version is made available under the CC-BY-NC-ND 4.0 license \href{https://creativecommons.org/licenses/by-nc-nd/4.0/}{https://creativecommons.org/licenses/by-nc-nd/4.0/}. This paper was not presented at any IFAC meeting. This work was supported by the Brazilian agencies Coordination for the Improvement of Higher Education Personnel (CAPES) – Finance Code 001, CNPq under grants 317058/2023-1 and 422143/2023-5, and FAPEMIG.}
\thanks{$^{1}$A. J. Santos and G. V. Raffo are with the Graduate Program in Electrical Engineering, Universidade Federal de Minas Gerais, Av. Antônio Carlos 6627, 31270-901, Belo Horizonte,
MG, Brazil.
{\footnotesize\tt \{aclecio1999, raffo\}@ufmg.br}}
\thanks{$^{2}$J. C. Pereira is with Department of Mechatronics Engineering, {\em Campus} Divin\'{o}polis --- CEFET--MG, Divin\'{o}polis, MG, 35503-822, Brazil} 
\thanks{$^{3}$ G. V. Raffo is also with the Department of Electronic Engineering, Universidade Federal de Minas Gerais.
        }
}

\begin{document}

\maketitle
\thispagestyle{empty}
\pagestyle{empty}

\begin{abstract}
This paper proposes an interior-point framework for constrained optimization problems whose decision variables evolve on matrix Lie groups. The proposed method, termed the Matrix Lie Group Interior-Point Method (MLG-IPM), operates directly on the group structure using a minimal Lie algebra parametrization, avoiding redundant matrix representations and eliminating explicit dependence on Riemannian metrics. A primal-dual formulation is developed in which the Newton system is constructed through sensitivity and curvature matrices. Also\GR{,} multiplicative updates are performed via the exponential map, ensuring intrinsic feasibility with respect to the group structure while maintaining strict positivity of slack and dual variables through a barrier strategy. A local analysis establishes quadratic convergence under standard regularity assumptions and characterizes the behavior under inexact Newton steps. Statistical comparisons against Riemannian Interior-Point Methods\GR{, specifically for} optimization problems defined over the Special Orthogonal Group $\mathrm{SO}(n)$ and  Special Linear Group $\mathrm{SL}(n)$\GR{, demonstrate that the proposed approach achieves} higher success rates, fewer iterations, \GR{and superior} numerical accuracy\GR{. Furthermore, its} robustness under perturbations suggests that \GR{this method serves as a consistent and reliable} alternative for structured manifold optimization.
\end{abstract}


\section{INTRODUCTION}

Optimization algorithms are widely used in robotics, control, \GR{and} machine learning, \GR{as they enable} the computation of optimal solutions for problems involving multiple criteria and constraints. In robotics and control \GR{specifically}, decision variables often represent configurations, poses, and spatial transformations of physical systems. In these settings, \GR{in addition to} minimizing a cost functional, it is necessary to \GR{satisfy} physical limits, safety requirements, and \GR{the underlying} system dynamics\GR{. This ensures} that the \GR{resulting} solution is not only optimal but also operationally feasible \cite{pedregal2004introduction}.

In many of these problems, the variables do not belong to $\mathbb{R}^n$, but rather to matrix Lie groups that model geometric transformations, such as rotations and rigid body motions. This \GR{is the case} when working with rotation matrices and homogeneous transformations, which must satisfy structural properties \GR{including} orthogonality, unit determinant, closure, and invertibility \cite{absil2008optimization}. When these problems are formulated in \GR{standard} Euclidean spaces, additional nonlinear constraints must be \GR{explicitly} imposed to \GR{ensure} the solution remains \GR{within} the matrix group\GR{. However, this approach increases} both \GR{the} analytical and computational \GR{complexities} \cite{trockman2021orthogonalizing}.

To \GR{account for} this geometric structure, \GR{specialized} approaches have been developed \GR{to} perform optimization \GR{directly} on differentiable manifolds and, in particular, on matrix Lie groups. \GR{The most prominent among these} are Riemannian optimization methods, which explicitly exploit the \GR{group's} geometry through tools such as the tangent space, Riemannian metrics, and the exponential map. Accordingly, the \GR{optimization} iterations are designed to remain intrinsically on the group, \GR{thereby eliminating} the need \GR{for} additional constraints to preserve algebraic properties \cite{hu2020brief}.

In this context, Riemannian methods can be \GR{categorized}, \GR{analogous} to the Euclidean case, into first- and second-order approaches, depending on whether they use only gradient information or incorporate \GR{second-order} curvature. Among first-order methods, Riemannian gradient descent stands out, where the descent direction is computed \GR{within} the tangent space, while the update is \GR{executed via} the exponential map or a retraction \GR{to ensure} the iterates remain on the group \cite{chen2021decentralized,bonnabel2013stochastic,samir2012gradient}. Second-order methods, such as the Riemannian Newton method, use the Riemannian Hessian, defined as the covariant derivative of the gradient, to \GR{exploit} curvature information and \GR{derive} more efficient search directions \cite{smith2014optimization}.

When \GR{optimization problems involve additional} equality and inequality constraints \GR{alongside the group structure}, extensions that combine Riemannian geometry with classical constrained optimization techniques are required. In this setting, Riemannian interior-point \GR{and barrier} methods are particularly relevant. These approaches \GR{formulate} barrier functions \GR{that} are compatible with the matrix Lie group structure, \GR{with} search directions computed \GR{within} the tangent space. \GR{Such} methods are \GR{critical for} control and \GR{motion} planning problems, \GR{where} the geometric \GR{integrity} of the Lie group \GR{must be preserved} while \GR{simultaneously} satisfying operational constraints in a systematic and numerically robust manner \cite{lai2024riemannian,nesterov2002riemannian,hirai2023interior}.

Despite their conceptual advantages, optimization methods on matrix Lie groups still present structural and numerical challenges. In particular, the matrix representation of the tangent space may introduce redundancies and linear dependencies associated with the chosen parametrization. Consequently, different matrices \GR{may} represent the same geometric direction, \GR{adversely affecting} numerical conditioning and computational efficiency \cite{bartelt2026constructivevectorfieldspath}. Moreover, existing Riemannian formulations inherently depend on the choice of a Riemannian metric, which determines the explicit form of the gradient and the Hessian. However, a bi-invariant Riemannian metric is not always available for a given group \cite{boumal2014manopt,milnor1976curvatures}. These limitations motivate the development of optimization algorithms that operate with minimal representations of matrix Lie groups, thereby \GR{eliminating} redundancies and reducing \GR{the} explicit dependence on metric structures and auxiliary geometric operations.

In this work, we \GR{introduce} the MLG-IPM (Matrix Lie Group Interior-Point Method), \GR{an optimizer} designed to solve constrained optimization problems whose decision variables belong to general matrix Lie groups. The proposed \GR{method} operates directly on the group structure and uses a minimal parametrization of the gradient\GR{. This approach effectively avoids the} redundancies \GR{inherent in standard} matrix representations and \GR{eliminates} the need to explicitly define \GR{a} Riemannian metric. In addition to the algorithmic \GR{framework, we provide} a local convergence analysis\GR{, establishing} theoretical guarantees for \GR{the method's} behavior in the neighborhood of \GR{an optimal} solution.

\textbf{Notation:} Scalars and vectors are denoted by lowercase letters, while stacked vectors composed of subvectors are denoted by bold lowercase letters. Matrices are denoted by bold uppercase letters, and $\mathbb{R}$ denotes the set of real numbers. Let $\mathrm{G}$ be a matrix Lie group with Lie algebra $\mathfrak{g}$, and let $\mathcal{S} : \mathbb{R}^m \to \mathfrak{g}$ be a linear map. Variations on the group are expressed via the exponential map $\exp(\cdot)$. We consider a tuple of matrices $\mathcal{X}\triangleq\{\mathbf{X}_1,\dots,\mathbf{X}_n\}$, with each $\mathbf{X}_i \in \mathrm{G}$. For a stacked vector $\boldsymbol{\zeta}\triangleq[\zeta_1^\top,\dots,\zeta_n^\top]$, where $\zeta_i \in \mathbb{R}^{m_i}$, we define
$
\mathcal{X}\boldsymbol{\exp}(\mathcal{S}(\boldsymbol{\zeta})) \triangleq
\{\mathbf{X}_1 \exp(\mathcal{S}(\zeta_1)),\dots,
\mathbf{X}_n \exp(\mathcal{S}(\zeta_n))\},$
which represents perturbations along the Lie algebra of each component. The identity matrix is denoted by $\mathbf{I}$, and $\operatorname{diag}(\cdot)$ forms a diagonal matrix from its arguments.

\section{Preliminaries}
This section establishes the mathematical foundations for the proposed framework. We review essential differential constructions on matrix Lie groups and derive first- and second-order linearizations via local Lie algebra coordinates. These derivations provide the sensitivity and curvature matrices that replace the classical Jacobian and Hessian.
\begin{definition}(\cite{hall2013lie})
A matrix Lie group is a subgroup $\mathrm{G}$ of $\mathsf{GL}(n;\mathbb{C})$ with the following property:
\GR{if} $\mathbf{A_m}$ is any sequence of matrices in $\mathrm{G}$, and $\mathbf{A_m}$ converges to some matrix $\mathbf{A}$,
then either $\mathbf{A}$ is in $\mathrm{G}$ or $\mathbf{A}$ is not invertible.
\end{definition}
\color{black}
\begin{definition} (\cite{bartelt2026constructivevectorfieldspath})
Let $\mathrm{G}$ be an $m$-dimensional Lie group and $\mathbf{G} \in \mathrm{G}$. 
Given a linear map $\mathcal{S} : \mathbb{R}^m \to \mathfrak{g}$ and a
differentiable function $f : \mathrm{G} \to \mathbb{R}$, 
the gradient $\mathcal{D}f_{[\mathbf{G}]}$ is defined by
\begin{align}
\mathcal{D}f_{[\mathbf{G}]}\boldsymbol{\zeta} {=} \lim_{\varepsilon \to 0} 
\frac{f\big(\mathbf{G} \exp(\mathcal{S}(\boldsymbol{\zeta})\,\varepsilon)\big) - f(\mathbf{G})}{\varepsilon}, \quad \varepsilon \in \mathbb{R},
\label{eq:derivative}
\end{align}
where $\boldsymbol{\zeta} \in \mathbb{R}^m$ denotes the generalized twist, and
$\setlength{\arraycolsep}{1.8pt}
\mathcal{D}f_{[\mathbf{G}]} = 
\begin{bmatrix} \ell_1(\mathbf{G}) & \ell_2(\mathbf{G}) & \cdots & \ell_m(\mathbf{G}) \end{bmatrix}, 
\,
\ell_m : \mathrm{G} \to \mathbb{R}.$
\end{definition}
\begin{remark}(\cite{bartelt2026constructivevectorfieldspath})
Equivalently, the gradient  $\mathcal{D}f_{[\mathbf{G}]}$  admits the differential
representation $\mathcal{D}f_{[\mathbf{G}]}\, \boldsymbol{\zeta} = \left.
\frac{d}{d\varepsilon}
f\!\big(\mathbf{G} \exp\!\big( \mathcal{S}(\boldsymbol{\zeta}) \, \varepsilon \big) \big)
\right|_{\varepsilon = 0}.$
\end{remark}
\begin{definition}
Let $\mathcal{G} = \{\mathbf{G}_1, \ldots, \mathbf{G}_n\}$ be a finite
collection of group elements.
The \emph{gradient} of the scalar function $f$ evaluated on $\mathcal{G}$ is
defined as $\setlength{\arraycolsep}{1.8pt}
\nabla_f^{\mathcal{G}}
=
\begin{bmatrix}
\mathcal{D}f_{[\mathbf{G}_1]}&
\mathcal{D}f_{[\mathbf{G}_2]} &
\cdots &
\mathcal{D}f_{[\mathbf{G}_n]}
\end{bmatrix}
\in \mathbb{R}^{1\times (n \cdot m)}.$
\end{definition}

\begin{definition}
Let $f = [f^1, \ldots, f^l]^\top :  \mathrm{G} \to \mathbb{R}^l$
be a differentiable mapping, where each component $f^j : \mathrm{G} \to \mathbb{R}$ is scalar-valued.
The Sensitivity Matrix of $f$ at $\mathcal{G}$ is defined as the matrix representation of the differential, 
$\boldsymbol{\mathfrak{J}}_{f}^\mathcal{G}
=
\begin{bmatrix}
\mathcal{D}f^1_{[\mathbf{G}_1]} & \cdots & \mathcal{D}f^1_{[\mathbf{G}_n]} \\
\vdots & \ddots & \vdots \\
\mathcal{D}f^l_{[\mathbf{G}_1]} & \cdots & \mathcal{D}f^l_{[\mathbf{G}_n]}
\end{bmatrix}
\in \mathbb{R}^{l\times (n \cdot m)}.
\label{eq:jacobian-vector}$
\end{definition}
\begin{definition}
Let $f : \mathrm{G} \to \mathbb{R}$ be twice differentiable.
The Curvature Matrix of $f$ at $\mathcal{G}$ is defined as the matrix representation of the second differential, namely
$\boldsymbol{\mathfrak{H}}^{\mathcal{G}}_f= \boldsymbol{\mathfrak{J}}_{(\nabla_f^\mathcal{G})^\top}^\mathcal{G},$
where $\boldsymbol{\mathfrak{H}}^{\mathcal{G}}_f \in \mathbb{R}^{(n \cdot m)\times (n \cdot m)}$.
\end{definition}
\begin{remark}
Although termed Sensitivity Matrix and Curvature Matrix, these objects correspond to the first- and second-order linearizations of the mapping in Lie algebra coordinates, playing roles analogous to the Jacobian and Hessian on $\mathbb{R}^n$ in Newton-type methods.
\end{remark}
\section{Matrix Lie Group Interior-Point Method}

This section \GR{details} the proposed interior-point framework formulated directly on matrix Lie groups. \GR{Drawing on the foundations in} \cite{lai2024riemannian,el1996formulation}\GR{, we} extend the classical primal--dual interior-point paradigm to optimization problems \GR{where} decision variables evolve on smooth matrix manifolds endowed with a Lie group structure.


Consider the nonlinear optimization problem
\begin{equation}
\begin{aligned}
&\min_{\mathcal{G}} \quad  f(\mathcal{G}) \\
&\text{s.t.}
 \qquad g_i(\mathcal{G}) \le 0, \quad i=1,\ldots,n_1, \\
 &\qquad\quad h_j(\mathcal{G}) = 0, \quad j=1,\ldots,n_2.
\end{aligned}
\label{eq:ipm-problem}
\end{equation}

The objective function $f : \mathcal{G} \to \mathbb{R}$ and the constraint \GR{mappings} $g_i : \mathcal{G} \to \mathbb{R}$ and $h_j : \mathcal{G} \to \mathbb{R}$ are assumed to be twice continuously differentiable. We further assume that the feasible set is nonempty and \GR{contains} at least one strictly feasible point \GR{such that} $g_i(\mathcal{G}) \le 0$ and $h_j(\mathcal{G}) = 0$ for all $i,j$.



\subsection{Lagrangian and primal--dual formulation}

\GR{The} primal-dual formulation \GR{provides a more operational perspective through the} Lagrangian associated with the problem~\eqref{eq:ipm-problem}, $\mathcal{L}(\mathcal{G},\nu,\lambda) = f(\mathcal{G}) + \nu^\top g(\mathcal{G}) + \lambda^\top h(\mathcal{G})$\GR{, with} $\nu \in \mathbb{R}^{n_1}$ and $\lambda \in \mathbb{R}^{n_2}$ \GR{denoting} the Lagrange multipliers associated with the inequality and equality constraints, respectively.


\GR{The} primal-dual interior-point framework \GR{reformulates} inequality constraints as equalities by introducing slack variables $s \in \mathbb{R}^{n_1}$, such that 
$g(\mathcal{G}) + s = 0, \; s_i > 0, i = 1,\cdots, n_1$.
This transformation \GR{enforces} strict feasibility \GR{through} the positivity of \GR{$s$. Consequently,  the} Karush–Kuhn–Tucker (KKT) conditions \GR{are defined}  by primal feasibility of the \GR{augmented equality system, dual feasibility, Lagrangian stationary on the Lie group, and} complementarity between slack and dual variables.


The stationarity condition is obtained by differentiating the Lagrangian with respect to $\mathcal{G}$:
$
\nabla_\mathcal{L}^{\mathcal{G}}
=
\nabla_f^\mathcal{G}
+ (\boldsymbol{\mathfrak{J}}_{g}^{\mathcal{G}})^\top \nu
+ (\boldsymbol{\mathfrak{J}}_{h}^\mathcal{G})^\top \lambda.
\label{eq:grad-lagrangian}$
Newton-type primal--dual methods \GR{require} second-order information\GR{,} provided by the \GR{Lagrangian} curvature matrix,
$\boldsymbol{\mathfrak{H}}^{\mathcal{G}}_\mathcal{L}
=
\boldsymbol{\mathfrak{H}}^{\mathcal{G}}_f
+ \sum_{i=1}^{n_1} \nu_i \boldsymbol{\mathfrak{H}}^{\mathcal{G}}_{g_i}
+ \sum_{j=1}^{n_2} \lambda_j \boldsymbol{\mathfrak{H}}^{\mathcal{G}}_{h_j}.
\label{eq:hessian-lagrangian}$
If the inequality constraints are linear, their second-order derivatives vanish. \GR{In such cases,} the Curvature Matrix coincides with that of the objective function augmented only by the second-order terms associated with the equality constraints.

Under these considerations, the primal-dual KKT conditions take the form
$\nabla_\mathcal{L}^\mathcal{G} = 0, \label{eq:kkt-stationarity} 
g(\mathcal{G}) + s
= 0, \label{eq:kkt-primal-g} 
h(\mathcal{G})
= 0, \label{eq:kkt-primal-h} 
\mathbf{S} \nu - \mu e
= 0,$
where $\mathbf{S} = \mathrm{diag}(s)$, $\mu > 0$ \GR{is} the barrier parameter, and $e$ is the vector of ones.
The classical complementarity condition $s_i \nu_i = 0$ is replaced by the relaxed relation $s_i \nu_i = \mu,$ which guarantees that all iterates remain strictly in the interior of the feasible region \cite{el1996formulation}.

\color{black}
To apply Newton’s method, we define the vector field
$\mathsf{F}
=\setlength{\arraycolsep}{2pt}
\begin{bmatrix}
(\nabla_\mathcal{L}^\mathcal{G})^\top &
(g(\mathcal{G}) + s)^\top&
h(\mathcal{G})^\top &
(\mathbf{S} \nu - \mu e)^\top
\end{bmatrix}^\top.$ The stacked primal--dual variable is defined as
$
\mathbf{z}
\triangleq
\begin{bmatrix}
\nu^\top &
\lambda^\top &
s^\top
\end{bmatrix}^\top \in \mathbb{R}^d$, with $d \triangleq p + 2m$. The primal--dual variables are then embedded into the affine matrix structure
$
\mathbf{Z}
\triangleq
\begin{bmatrix}
\mathbf{I_d} & \mathbf{z} \\
\mathbf{0}_{1\times d} & 1
\end{bmatrix}
\in \mathsf{T}(d),$
where $\mathsf{T}(d)$ denotes the $d$-dimensional translation group, represented by affine matrices whose action corresponds to translations in $\mathbb{R}^d$ \cite{bartelt2026constructivevectorfieldspath}. This representation allows multiplicative updates through the exponential map while preserving positivity and the complementarity structure.

Linearizing the KKT conditions around the current iterate $\mathcal{X}^k = \{\mathcal{G}^k,\mathbf{Z}^k\}$ leads to the primal--dual Newton system
\begin{align}
\boldsymbol{\mathfrak{J}}^{\mathcal{X}^k}_{\mathsf{F}} \cdot \boldsymbol{\Delta x^k}
=- \mathsf{F}(\mathcal{X}^k),
\label{EQ:PRIMAL_DUAL}
\end{align}
where the search direction is defined as
$\boldsymbol{\Delta x^k} =[(\Delta x_1^k)^\top,\cdots,(\Delta x_n^k)^\top,(\Delta x_{n+1})^\top]^\top
\in \mathbb{R}^{nm+d}$, with $\Delta x_{1:n}^k \in \mathbb{R}^{m} $, $\Delta x_{n+1}^k \in \mathbb{R}^{d} $, and $\boldsymbol{\mathfrak{J}}^{\mathcal{X}^k}_{\mathsf{F}}$ evaluated at $\mathcal{X}$ \GR{being} given by
$
\begingroup
\setlength{\arraycolsep}{0.1pt}
\boldsymbol{\mathfrak{J}}^{\mathcal{X}^k}_{\mathsf{F}}=
\begin{bmatrix}
\boldsymbol{\mathfrak{H}}^{\mathcal{G}^k}_\mathcal{L}
& (\boldsymbol{\mathfrak{J}}_g^{\mathcal{G}^k})^\top
& \boldsymbol{\mathfrak{J}}_h^{\mathcal{G}^k})^\top
& 0 \\
\boldsymbol{\mathfrak{J}}_g^{\mathcal{G}^k}
& 0
& 0
& I \\
\boldsymbol{\mathfrak{J}}_h^{\mathcal{G}^k}
& 0
& 0
& 0 \\
0
& \mathbf{S}^k
& 0
& \mathbf{\mathrm{diag}}(\nu^k)
\end{bmatrix}.
\endgroup
$

After solving the Newton system, the primal variables evolve on the Lie group according to
\begin{equation}
\mathcal{X}^{k+1}
=
\mathcal{X}^{k}\boldsymbol{\exp}\big(\mathcal{S}(\alpha\boldsymbol{\Delta x^k)}\big),
\label{eq:passo_G}
\end{equation}
where the update is applied componentwise, with 
$\mathcal{X}^{k}=\{\mathbf{G}_1^{k},\cdots,\mathbf{G}_n^{k},\mathbf{Z}^{k}\}$.


\color{black}

The step size $\alpha \in (0,1]$ is determined using a fraction-to-the-boundary strategy in order to preserve strict positivity of the slack and dual variables. To this end, we compute the maximum admissible primal and dual step lengths as
$\alpha_{\mathrm{pri}}
=
\min_{i:\,\Delta s_i<0}
\left\{
-\frac{s_i}{\Delta s_i}
\right\},
\,
\alpha_{\mathrm{dual}}
=
\min_{i:\,\Delta \nu_i<0}
\left\{
-\frac{\nu_i}{\Delta \nu_i}
\right\},$
where the minimum over an empty index set is interpreted as $+\infty$. The actual step size is then defined by
$
\alpha
=
\tau\,
\min\!\big(1,\alpha_{\mathrm{pri}},\alpha_{\mathrm{dual}}\big),
$
with $\tau \in (0,1)$ ensuring that $s^{k+1} > 0$ and $\nu^{k+1} > 0$.


The complementarity parameter is then updated according to
$
\mu
=
\sigma\,\frac{s^\top \nu}{m},
$
with $\sigma \in [0.1,\,0.5]$. The iterations are repeated until $\mu$ becomes sufficiently small. The complete MLG-IPM procedure is summarized in Algorithm~\ref{alg:mlg-ipm}.
\begin{remark}
\GR{Multiplicative updates via the exponential map ensure that iterates remain on the Lie group at every step. In contrast to purely additive first-order schemes \cite{polik2010interior,vanderbei1999interior,lee2012exponential}, this approach inherently incorporates higher-order geometric information derived from the group structure.}
\end{remark}
\color{black}

\begin{algorithm}[t]
\caption{Matrix Lie Group Interior-Point Method}
\label{alg:mlg-ipm}
\begin{algorithmic}[1]
\STATE \textbf{Input:} Initial strictly feasible $\mathcal{X}^0$ with $s^0>0$, $\nu^0>0$;
tolerance $\varepsilon_{\text{tol}}>0$; parameters $\sigma \in (0.1,0.5)$, $\tau \in (0,1)$.
\STATE Set $k \gets 0$.
\WHILE{$\|\mathsf{F}(\mathcal{X}^k)\| > \varepsilon_{\text{tol}}$}
\STATE Compute  $\mathsf{F}^k$.
\STATE Form the Newton matrix $\boldsymbol{\mathfrak{J}}^{\mathcal{X}^k}_{\mathsf{F}}$.
\STATE Solve the linear system
$
\boldsymbol{\mathfrak{J}}^{\mathcal{X}^k}_{\mathsf{F}}\boldsymbol{\Delta x^k} = - \mathsf{F}^k$.
\STATE Compute maximum feasible step sizes:
$
\alpha_{\mathrm{pri}},
\alpha_{\mathrm{dual}}$
\STATE Set
$
\alpha^k
=
\tau\,
\min\!\big(1,\alpha_{\mathrm{pri}},\alpha_{\mathrm{dual}}\big).$
\STATE Update primal Lie group variables:
$
\mathcal{X}^{k+1}
=
\mathcal{X}^{k}
\boldsymbol{\exp}\big(\mathcal{S}(\alpha^k \Delta x^k)\big).$
\STATE Recover $\mathcal{X}^k$.
\STATE Update barrier parameter:
$
\mu^{k+1}
=
\sigma \frac{(s^{k+1})^\top \nu^{k+1}}{m}.
$
\STATE $k \gets k+1$.
\ENDWHILE
\STATE \textbf{Output:} $\mathcal{X}^k$.

\end{algorithmic}
\label{alorithm1}
\end{algorithm}

\section{Local Convergence Analysis}

\GR{This section presents the} local convergence analysis \GR{for} the proposed Newton-type method \GR{on} $\mathcal{X}$. \GR{Drawing on the} frameworks in \cite{lai2024riemannian,el1996formulation}, \GR{we suitably adapt classical results to accommodate the primal-dual formulation and the underlying Lie group geometry. All} assumptions and convergence results are reformulated to be consistent with the \GR{geometric} notation \GR{established} throughout \GR{this work}.

\begin{assumption}
There exists a  $\mathcal{X}^*$ satisfying the KKT conditions, i.e.,
$\mathsf{F}(\mathcal{X}^*) = 0$.
\label{ass:a1}
\end{assumption}
\begin{assumption}
The set
$
\{ \nabla_{h_j}^{\mathcal{G}^*}\}_{j=1}^{n_2}
\cup
\{\nabla_{g_i}^{\mathcal{G}^*}\}_{i\in \mathbb{A}(\mathcal{G}^*)}
$
is linearly independent, where the active set is defined as
$
\mathbb{A}(\mathcal{G}^*) := \{ i \mid g_i(\mathcal{G}^*) = 0 \}.
$
\label{ass:a2}
\end{assumption}

\begin{assumption}
$
\nu^*_i > 0
\,
\text{whenever}
\,
g_i(\mathcal{G}^*) = 0,
\;i=1,\dots,n_1 .
$
\label{ass:a3}
\end{assumption}

\begin{assumption}
The \GR{Lagrangian} curvature matrix satisfies
$
\boldsymbol{\zeta}^{\top}
\boldsymbol{\mathfrak{H}}^{\mathcal{G}^*}_{\mathcal{L}}
\boldsymbol{\zeta} > 0
$
for every nonzero vector 
$\boldsymbol{\zeta} \in \mathbb{R}^{n\cdot m}$ such that
$
\nabla_{h_j}^{\mathcal{G}^*}\boldsymbol{\zeta} = 0,
\; j=1,\dots,n_2,
$
and
$
\nabla_{g_i}^{\mathcal{G}^*}\boldsymbol{\zeta} = 0,
\;i\in \mathbb{A}(\mathcal{G}^*).
$
\label{ass:a4}
\end{assumption}
\begin{theorem}{(\cite{lai2024riemannian})}
If \GR{Assumptions} \ref{ass:a1}\,--\,\ref{ass:a4} hold at some point $\mathcal{G}^*$, 
then the Sensitivity Matrix $\boldsymbol{\mathfrak{J}}^{\mathcal{X}^*}_{\mathsf{F}}$ is nonsingular.
\label{th:t1}
\end{theorem}
\begin{theorem}
\GR{Under} Assumptions \ref{ass:a1}-\ref{ass:a4}\GR{, } there exists a neighborhood $\mathcal{U}$ of $\mathcal{X}^*$ such that, for any initial point $\mathcal{X}^0 \in \mathcal{U}$ sufficiently close to $\mathcal{X}^*$, the  iteration
\begin{align}
\boldsymbol{\mathfrak{J}}^{\mathcal{X}^k}_{\mathsf{F}}\boldsymbol{\Delta x^k}
=
- \mathsf{F}(\mathcal{X}^k) + \rho_k,
\label{eq:wp}
\end{align}
combined with the exponential update \eqref{eq:passo_G}, is well defined and generates a sequence $\{\mathcal{X}^k\}$ converging to $\mathcal{X}^*$. The local convergence rate \GR{is governed by} the step \GR{size} $\alpha$ and residual $\rho_k$ \GR{as follows:}

{For ${\alpha{=}1}$:} 
\begin{enumerate}
\item[(i)] Quadratic convergence if $\rho_k{=}0$ or $\|\rho_k\| {=} O(\|\mathsf{F}(\mathcal{X}^k)\|^2)$;
\item[(ii)] Linear convergence if $\|\rho_k\| {=} O(\|\mathsf{F}(\mathcal{X}^k)\|)$.
\end{enumerate}

For ${0{<}\alpha{<}1}$:
\begin{enumerate}
\item[(i)] Linear convergence if $\rho_k{=}0$ or $\|\rho_k\| {=} O(\|\mathsf{F}(\mathcal{X}^k)\|)$.
\end{enumerate}

\label{th:t2}
\end{theorem}
\begin{proof}
\GR{Applying \eqref{eq:derivative} to the vector field $\mathsf{F}$ yields}
\begin{equation}
\mathsf{F}\Big(\mathcal{X}^{k}\boldsymbol{\exp}(\mathcal{S}(\boldsymbol{\zeta})\varepsilon)\Big) 
= \mathsf{F}(\mathcal{X}^k)
+ \boldsymbol{\mathfrak{J}}_\mathsf{F}^\mathcal{X} \boldsymbol{\zeta}\,\varepsilon 
+ O(\|\boldsymbol{\zeta}\,\varepsilon \|^2),
\label{eq:gradient_expansion}
\end{equation}
where 
$\boldsymbol{\zeta} = [\zeta_1^\top, \cdots, \zeta_n^\top, \zeta_{n+1}^\top]^\top \in \mathbb{R}^{nm+d}$, with $\zeta_{1:n} \in \mathbb{R}^{m} $ and $\zeta_{n+1} \in \mathbb{R}^{d} $. \GR{By setting $\boldsymbol{\zeta}\,\varepsilon = \alpha \boldsymbol{\boldsymbol{\Delta x^k}}$ and noting that $\mathcal{X}^{k+1} = \mathcal{X}^k \exp(\mathcal{S}(\alpha \boldsymbol{\Delta x^k}))$}, \eqref{eq:gradient_expansion} reduces to
\begin{equation}
\begin{aligned}\boldsymbol{}
\mathsf{F}(\mathcal{X}^{k+1})
{=} \mathsf{F}(\mathcal{X}^k)
{+} \alpha \boldsymbol{\mathfrak{J}}^{\mathcal{X}^k}_{\mathsf{F}}\,\GR{\boldsymbol{\Delta x^k}}
{+} O(\|\alpha \GR{\boldsymbol{\Delta x^k}\|^2)}.
\label{eq:taylor_local}
\end{aligned}
\end{equation}
\GR{Substituting} the Newton direction \GR{from} \eqref{eq:wp} into \eqref{eq:taylor_local}\GR{, we obtain}
\begin{equation}
\mathsf{F}(\mathcal{X}^{k+1}) = (1{-}\alpha)\mathsf{F}(\mathcal{X}^k) {+}\alpha \rho_k + O(\|\alpha \GR{\boldsymbol{\Delta x^k}}\|^2).
\label{eq:residual_update}
\end{equation}
From Theorem~\ref{th:t1}, \GR{the nonsingularity of} $\boldsymbol{\mathfrak{J}}^{\mathcal{X}}_{\mathsf{F}}$ in \GR{the} neighborhood of $\mathcal{X}^*$ \GR{ensures that} $(\boldsymbol{\mathfrak{J}}^{\mathcal{X}}_{\mathsf{F}})^{-1}$ is bounded. \GR{This leads to the estimate}
\begin{equation}
\|\mathsf{F}(\mathcal{X}^{k+1})\|
{\le}
C_1 \|\rho_k\|
{+}
(1{-}\alpha)\|\mathsf{F}(\mathcal{X}^{k})\|
{+}
C_2\|\mathsf{F}(\mathcal{X}^{k})\|^2,
\label{eq:bound}
\end{equation}
for some constants $C_1,C_2>0$. \GR{The convergence behavior is determined by the parameters $\alpha$ and $\rho_k$}
\begin{itemize}

\item \GR{{Case $\alpha = 1$:} The linear term vanishes and \eqref{eq:bound} reduces to
$\|\mathsf{F}(\mathcal{X}^{k+1})\|
\le
C_1\|\rho_k\|
+
C_2\|\mathsf{F}(\mathcal{X}^{k})\|^2$. In this case, the behavior of the method is mainly governed by the residual $\rho_k$. Quadratic convergence is achieved if $\|\rho_k\| = O(\|\mathsf{F}(\mathcal{X}^{k})\|^2)$ (including $\rho_k=0$), whereas $\|\rho_k\| = O(\|\mathsf{F}(\mathcal{X}^{k})\|)$ results in linear convergence. }


\item \GR{{Case $0 < \alpha < 1$:} The term $(1-\alpha)\|\mathsf{F}(\mathcal{X}^{k})\|$ introduces a linear component in the estimate, which dominates the asymptotic behavior unless the residual $\rho_k$ decays sufficiently fast. In particular, the method exhibits linear convergence when $\|\rho_k\|=O(\|\mathsf{F}(\mathcal{X}^{k})\|)$ or $\|\rho_k\| = 0$.}
\end{itemize}
\GR{Finally, the nonsingularity of $\boldsymbol{\mathfrak{J}}^{\mathcal{X}}_{\mathsf{F}}(\mathcal{X}^*)$ ensures that $\|\mathsf{F}(\mathcal{X}^{k})\|\to 0$ that implies $\mathcal{X}^{k}\to \mathcal{X}^*$, completing the proof}.
\end{proof}
\section{Comparative Study}
\label{sec:statistical-analysis}
\GR{This section evaluates the performance of the proposed MLG-IPM through a quantitative comparison with the Riemannian Interior-Point Method (RIPM) \cite{lai2024riemannian} and a qualitative assessment against} a Euclidean Interior-Point Method (EIPM) \cite{el1996formulation}. \GR{The} quantitative analysis \GR{involved} $1000$ independent simulations on $\mathsf{SO}(7)$ and $\mathsf{SL}(7)$ \GR{per method, measuring success rate, computational time, number of iterations, and residual error. Because the Lilliefors test ($\gamma = 0.05$)} rejected normality for the continuous metrics, \GR{we employed} the nonparametric Wilcoxon–Mann–Whitney test \GR{for comparisons}. \GR{Categorical outcomes were analyzed via} the chi-square test and the two-proportion Z-test \GR{for} the null hypothesis $H_0: p_{\text{MLG}} = p_{\text{RIPM}}$, where $p_{\text{MLG}}$ and $p_{\text{RIPM}}$ denote the success probabilities of MLG-IPM and RIPM, respectively. \GR{Finally, effect} sizes were quantified using Cramér’s $V$ for categorical variables and Cohen’s $d$ for continuous metrics. 

\subsection{Comparison between RIPM and MLG-IPM in $\mathsf{SO}(7)$}
The optimization problem considered is
\begin{equation}
\begin{aligned}
&\min_{\mathbf{G_1},\mathbf{G_2} \in \mathsf{SO}(7)} \; 
\frac{1}{2}\sum_{i=1}^{2}\|\log(\mathbf{G_d}^\top \mathbf{G}_i)\|_F \\
&\text{s.t.} \; 
\mathbf{G}_i(1,1) - 0.5 \le 0, 
 \mathbf{G}_i(2,2) - 0.3 \le 0,
\end{aligned}
\label{eq:p1}
\end{equation}
with $\mathbf{G_1}$, $\mathbf{G_2}$, and $\mathbf{G_d}$ randomly sampled in $\mathsf{SO}(7)$ at each run.

\subsubsection{Without Perturbation}

\begin{table}[h]
\setlength{\tabcolsep}{3pt}
\centering
\caption{Statistical comparison between MLG-IPM and RIPM with $\rho_k = 0$ in $\mathsf{SO}(7)$.}
\label{tab:comparacao1}
\begin{tabular}{lcc}
\hline
\textbf{Metric} & \textbf{MLG-IPM} & \textbf{RIPM} \\
\hline
Time (s) Median      & 10.5661          & 20.8154          \\
Time (s) Mean   & $13.6649 \pm 25.2225$ & $18.4179 \pm 7.4139$ \\
Iterations Median     & 20.00          & 500.00         \\
Iterations Mean   & $25.76 \pm 47.84$ & $394.31 \pm 139.98$ \\
Success Rate            & 0.991 & 0.416          \\
Error Mean   & $5.0 {\times}10^{-5} {\pm} 4.17 {\times} 10^{-4} $ & $4.7 {\times} 10^{-3}{\pm} 1.7 {\times} 10^{-2}$ \\
\hline
\end{tabular}
\end{table}

Table~\ref{tab:comparacao1} summarizes the performance of both methods for $\rho_k = 0$, revealing not only statistical differences but also meaningful practical implications. 
\GR{The success rates highlight a fundamental gap in robustness: MLG-IPM converges in nearly all executions, whereas RIPM fails in a significant portion of cases. The strong association between optimizer choice and convergence (Cramér's $V = 0.628$) confirms that the probability of success is highly dependent on the selected method.}

\GR{Computationally, MLG-IPM achieves solutions faster and with fewer iterations. Conversely, the high frequency with which RIPM reaches the iteration limit suggests stagnation near the stopping threshold rather than stable convergence. This indicates a structural difference in behavior: while MLG-IPM progresses steadily toward optimality, RIPM often fails to satisfy termination criteria within reasonable bounds.}

\GR{Regarding accuracy, MLG-IPM yields a lower mean final error across successful runs, indicating convergence closer to the optimal solution. The larger errors and greater dispersion observed for RIPM reflect reduced numerical precision. Collectively, these results establish the superiority of MLG-IPM across three dimensions: reliability (success rate), efficiency (computational cost), and accuracy (final residual). For $\rho_k = 0$, the MLG-IPM advantages are both statistically significant and structurally decisive.}


\subsubsection{With Perturbation}

\begin{table}[h]
\centering
\setlength{\tabcolsep}{3pt}
\caption{Statistical comparison between MLG-IPM and RIPM with $\rho_k = 0.01\|\mathsf{F}^k\|$ in $\mathsf{SO}(7)$ .}
\label{tab:comparison2}
\begin{tabular}{lcc}
\hline
\textbf{Metric} & \textbf{MLG-IPM} & \textbf{RIPM} \\
\hline
Time (s) Median      & 11.44          & 22.29          \\
Time (s) Mean    & $18.88 \pm 40.89$ & $18.94 \pm 7.50$ \\
Iterations Median     & 21.00          & 500.00         \\
Iterations Mean   & $35.12 \pm 77.52$ & $393.45 \pm 141.44$ \\
Success Rate            & 0.973 & 0.388          \\
Error Mean          & $4.940 {\times} 10^{-3} {\pm} 1.0 {\times} 10^{-6}$ & $9.9 {\times} 10^{-3} {\pm}1.7 {\times} 10^{-2}$       \\
\hline
\end{tabular}
\end{table}

Table~\ref{tab:comparison2} reports the results obtained with residual based perturbation, $\rho_k = 0.01\|\mathsf{F}^k\|$. Although the regularization slightly modifies the numerical behavior of both methods, the overall performance hierarchy remains unchanged.

The success rate continues to favor MLG-IPM by a wide margin, indicating that the perturbation does not significantly alter its robustness. The \GR{chi-square} and \GR{Z- tests} confirm statistical significance ($p < 0.001$), and the large effect size (Cramér's $V = 0.626$) reinforces that optimizer choice remains strongly associated with convergence outcome even under regularization.

In terms of computational effort, the Wilcoxon-Mann-Whitney tests again reject $H_0$ for both time and iterations. As shown in Table~\ref{tab:comparison2}, MLG-IPM maintains lower median time and substantially fewer iterations. Although its mean time presents higher dispersion, this behavior stems from occasional longer runs rather than systematic inefficiency. In contrast, RIPM continues to accumulate results at the maximum iteration limit, indicating persistent structural difficulty in meeting stopping criteria despite the perturbation.

Regarding solution quality, the final error values confirm that MLG-IPM preserves superior numerical accuracy among successful runs. The Wilcoxon-Mann-Whitney  test detects statistical significance, and the difference in mean error suggests that residual based regularization does not compensate for the intrinsic convergence limitations observed in RIPM.

\subsection{Comparison between RIPM and MLG-IPM in $\mathsf{SL}(7)$}

For the comparison in $\mathsf{SL}(7)$, we consider the following optimization problem:
\begin{equation}
\begin{aligned}
&\min_{\mathbf{G_1},\mathbf{G_2} \in \mathsf{SL}(7)} \; 
\frac{1}{2}\sum_{i=1}^{2}\text{Tr}(\mathbf{G_d}^{-1} \mathbf{G}_i)^2 \\
&\text{s.t.} \;
-\text{Tr}(\mathbf{G_d}^{-1} \mathbf{G}_i)^2+0.2\le 0, \\
\end{aligned}
\label{eq:p2}
\end{equation}
with $\mathbf{G_1}$, $\mathbf{G_2}$, and $\mathbf{G_d}$ randomly sampled in $\mathsf{SL}(7)$ at each run.

\subsubsection{Without Perturbation}
\begin{table}[h]
\setlength{\tabcolsep}{3pt}
\centering
\caption{Statistical comparison between MLG-IPM and RIPM with $\rho_k = 0$ in $\mathsf{SL}(7)$.}
\label{tab:comparacao3}
\begin{tabular}{lcc}
\hline
\textbf{Metric} & \textbf{MLG-IPM} & \textbf{RIPM} \\
\hline
Time (s) Median      & 0.1187          & 6.6486          \\
Time (s) Mean   & $0.2590 \pm 0.4097$ & $7.60 \pm 7.83$ \\
Iterations Median     & 58          & 500         \\
Iterations Mean   & $66.98 \pm 63.31$ & $369 \pm 188.1$ \\
Success Rate            & 0.981 & 0.359          \\
Error Mean   & $6.6 {\times}10^{-5} \pm 2.0 {\times}10^{-3}$ & $2.7 {\times}10^{-3} \pm 1.4 {\times}10^{-2}$ \\
\hline
\end{tabular}
\end{table}

Table~\ref{tab:comparacao3} shows that MLG-IPM converged in 98.1\% of executions, whereas RIPM succeeded in only 35.9\% (an absolute difference of 62.2 percentage points). The high Cramér's $V$ ($V = 0.660$) indicates a strong association between optimizer choice and convergence probability, highlighting a substantial robustness advantage for MLG-IPM.

To formally assess the difference in success proportions, we performed a chi-square test, yielding $\chi^2 = 872.1$. Although the reported $p$ value was $p = 1$, the Z-test for proportions resulted in $Z = 29.6$ with $p < 0.05$, indicating a statistically significant difference. The extremely large Z statistic reinforces that the observed discrepancy is unlikely due to random variation, consistent with the large effect size measured by Cramér's $V$.

Computationally, both time and iteration distributions were identified as non normal (Lilliefors test, $p < 0.05$), so we conducted comparisons using the Wilcoxon-Mann-Whitney test. The differences were statistically significant ($p < 0.05$) for both metrics. The median execution time of MLG-IPM (0.1187~s) is \GR{much} smaller than that of RIPM (6.6486~s), yielding a median difference of approximately 6.53 seconds. Similarly, the median iteration count (58 vs.\ 500) indicates that RIPM frequently reaches the maximum iteration limit, suggesting stagnation rather than natural convergence.

Regarding solution accuracy, considering only successful runs, MLG-IPM achieved a substantially smaller mean final error ($6.6 \times 10^{-5}$) compared to RIPM ($2.68 \times 10^{-3}$). The Wilcoxon\GR{-Mann-Whitney} test confirmed statistical significance ($p < 0.05$). Moreover, RIPM exhibits greater error dispersion, indicating reduced numerical consistency.

\subsubsection{With Perturbation}
\begin{table}[h]
\setlength{\tabcolsep}{3pt}
\centering
\caption{Statistical comparison between MLG-IPM and RIPM with $\rho_k = 0.001\|\mathsf{F}^k\|$ in $\mathsf{SL}(7)$.}
\label{tab:comparacao4}
\begin{tabular}{lcc}
\hline
\textbf{Metric} & \textbf{MLG-IPM} & \textbf{RIPM} \\
\hline
Time (s) Median      & 0.1346          & 6.8945          \\
Time (s) Mean   & $0.36 \pm 0.60$ & $6.28 \pm 6.77$ \\
Iterations Median     & 62          & 500         \\
Iterations Mean   & $106.15 \pm 123.96$ & $344.54 \pm 197.3$ \\
Success Rate            & 0.937 & 0.408          \\
Error Mean   & $5 \times 10^{-4} \pm 5 \times 10^{-3}$ & $2 \times 10^{-3} \pm 1 \times 10^{-2}$ \\
\hline
\end{tabular}
\end{table}

Table~\ref{tab:comparacao4} presents the statistical comparison between MLG-IPM and RIPM based on 1000 independent runs per method for problem \eqref{eq:p2}. The results reveal statistically significant and practically meaningful differences across all evaluated metrics.

In terms of reliability, MLG-IPM converged in 93.7\% of executions, whereas RIPM succeeded in only 40.8\% . The chi-square test gave $\chi^2 = 632.90$, and although the reported $p$ value was $p = 1.00$, the Z-test for proportions yielded $Z = 25.21$ with $p < 0.05$, indicating a statistically significant difference. Furthermore, Cramér's $V = 0.563$ indicates a large effect size, confirming a strong association between optimizer choice and convergence probability.

Computationally, the Lilliefors test indicated that both time and iteration distributions deviate from normality ($p < 0.05$), so we used the Wilcoxon-Mann-Whitney test. The difference in execution time was statistically significant ($p < 0.05$), with MLG-IPM presenting a substantially lower median time (0.1346 s) compared to RIPM (6.8945 s).

A similar pattern emerged for iteration count. MLG-IPM required significantly fewer iterations (median of 62) than RIPM (median of 500), which frequently reached the maximum iteration limit, suggesting stagnation rather than natural convergence. The Wilcoxon test confirmed statistical significance ($p < 0.05$), with a median difference of 438 iterations.

Regarding numerical accuracy, considering only successful runs, MLG-IPM achieved a lower mean final error ($5.22 \times 10^{-4}$) compared to RIPM ($1.67 \times 10^{-3}$). This difference was also statistically significant according to the Wilcoxon\GR{-Mann-Whitney} test ($p < 0.05$). Additionally, RIPM exhibited greater error dispersion, indicating lower numerical consistency.

\subsection{Comparison between EIPM and MLG-IPM}

We compare the EIPM \cite{el1996formulation} and the proposed MLG-IPM through a structural and algorithmic analysis. The comparison focuses on three main aspects: the underlying geometric framework, the treatment of structural constraints, and the computational implications of solving the Newton-type systems arising at each iteration.

\paragraph{Geometric setting}
EIPM is formulated in Euclidean space, where structural constraints are imposed explicitly through equality conditions. In contrast, MLG-IPM operates directly on a matrix Lie group, treating the feasible set as a smooth manifold endowed with group structure.

\paragraph{Treatment of constraints}
In EIPM, properties such as orthogonality \GR{and} determinant conditions are enforced via equality constraints and Lagrange multipliers. In MLG-IPM, these properties are intrinsic to the Lie group and \GR{are therefore} incorporated directly into the search space. This leads to a more compact problem formulation.

\paragraph{Computational dimension}
Both methods require the solution of a Newton-type linear system at each iteration. In EIPM, the system dimension includes primal variables and multipliers associated with equality constraints. In MLG-IPM, the search direction is computed in the Lie algebra, whose dimension equals that of the group. When a group structure replaces explicit constraints, this may reduce the dimension of the linear system and potentially improve computational efficiency.

\paragraph{Modeling capability}
Problems defined over structured matrix sets can be difficult to treat in EIPM due to nonlinear equality constraints. MLG-IPM handles such problems naturally by preserving the group structure throughout the iterations.

\section{Conclusion}

This work introduced MLG-IPM as an interior-point method formulated directly on matrix Lie groups through a minimal Lie algebra parametrization, providing a geometrically consistent and compact framework for constrained optimization on non-Euclidean manifolds. By computing search directions in local exponential coordinates and performing multiplicative updates that preserve the intrinsic structure of the group, the method avoids redundant parametrizations and explicit metric constructions while maintaining the main theoretical properties of primal--dual interior-point schemes. The local convergence analysis \GR{established} quadratic convergence under standard assumptions and \GR{characterized} the influence of inexact Newton steps on the convergence rate. Numerical experiments \GR{indicated} that MLG-IPM achieves higher success rates, fewer iterations, and improved numerical accuracy when compared to RIPM, while maintaining robustness under perturbations. These results suggest that MLG-IPM is a practical approach for constrained optimization on matrix Lie groups, with potential applications in robotics, geometric control, and estimation on nonlinear manifolds. Future work will focus on establishing global convergence results, developing algorithmic improvements to enhance scalability for large-scale problems, and exploring applications in robotics and geometric control, where optimization on matrix Lie groups naturally arises.

\bibliographystyle{IEEEtran}
\bibliography{reference}

\end{document}